\documentclass[12pt,leqno]{article}

\usepackage{amsmath, amsfonts, amssymb}
\usepackage{mathrsfs}
\usepackage{theorem}
\usepackage{bm}

\newtheorem{thm}{Theorem}[section]
\newtheorem{prop}[thm]{Proposition}
\newtheorem{lem}[thm]{Lemma}
\newtheorem{cor}[thm]{Corollary}
\newtheorem{defn}[thm]{Definition}

{\theorembodyfont{\rmfamily} \newtheorem{rem}[thm]{Remark}}
{\theorembodyfont{\rmfamily} }

\newcommand{\ra}{\rightarrow}
\newcommand{\dis}{\displaystyle}

\def\R{\mathbb R}

\def\N{\mathbb N}

\def\d{\text{\rm{d}}}
\def\E{\mathbb E}
\def\p{\mathbb P}\def\e{\text{\rm{e}}}

\def\la{\langle}
\def\raa{\rangle}
\def\La{\Lambda}
\def\veps{\varepsilon}

\def\S{\mathcal S}

\newcommand{\fin}{\hspace*{\fill}\rule{0.3em}{1ex}}
\newenvironment{proof}{{\bf \noindent Proof.}}{\fin}

\numberwithin{equation}{section}

\allowdisplaybreaks

\begin{document}

\title{Heavy tail and light tail of Cox-Ingersoll-Ross processes with regime-switching\footnote{Supported in
 part by    NNSFs of China No. 11431014, 11771327}}

\author{Tongtong Hou${}^b$,\ \ Jinghai Shao${}^{a}$\footnote{Corresponding author, Email: shaojh@bnu.edu.cn}\\[0.2cm]
{\small a: Center for Applied Mathematics, Tianjin
University, Tianjin 300072, China.}\\
{\small b: School of Mathematical Sciences, Beijing Normal University, Beijing 100875, China.}}
\maketitle

\begin{abstract}
  This work is denoted to studying the tail behavior of Cox-Ingersoll-Ross (CIR) processes with regime-switching. One essential difference shown in this work between CIR process with regime-switching and without regime-switching is that the stationary distribution for CIR process with regime-switching could be heavy-tailed.  Our results provide a theoretical evidence of the existence of regime-switching for interest rates model based on its heavy-tailed empirical evidence. In this work, we first provide sharp criteria to justify the existence of stationary distribution for the CIR process with regime-switching, which is applied to study the long term returns of interest rates. Then under the existence of the stationary distribution, we provide a criterion to justify whether its stationary distribution is heavy-tailed or not.
\end{abstract}

\noindent AMS subject Classification:\ 60J60, 60H10, 60H30

\noindent\textbf{Key words}: Regime-switching, Heavy tail, Light tail, Ergodicity

\section{Introduction}
Modelling the term structure of interest rates is a long-standing topic in financial economics. Many stochastic interest rates models have been proposed in the past several decades in order to provide a realistic and tractable method to describe the term structure. Some early contributions include Vasieck \cite{Va}, Dothan \cite{Do}, Cox et al. \cite{CIR} and Hull and White \cite{HW}, amongst others. Single-factor term  structure models have been extended to multi-factor ones in the literatures, for instance, Longstaff and Schwartz \cite{LS}, Duffie and Kan \cite{DK}.

Regime-switching models have emerged in many research fields such as biological, ecological, mathematical finance, economics, etc. Early applications of regime-switching models to economics include Hamilton \cite{Ham}, and Garcia and Perron \cite{GP}. Regime-switching behavior of interest rate models have been used in interest rates modelling. Empirical evidence provided in the finance literatures Aug and Bekaert \cite{ABa, ABb} suggests that the switching of regimes in interest rates matches well with business cycles. In addition to the statistical evidence, there are economic reasons
as well to believe that the regime shifts are important to understand the behavior of entire process.

The standard Cox-Ingersoll-Ross model does not consider the possibility of changes in regime.  As shown in Brown and Dybvig \cite{BD}, based on the empirical data of US. treasure yields, the poor empirical performance of CIR model may well suggest the existence of regime shifts. Hence, Gray in \cite{Gray} models the short interest rate as a discrete regime-switching process. Elliott and Siu \cite{ES} proposed a model of term structure of interest rates in a Markovian, regime-switching Health-Jarrow-Morton framework. Zhang et al. \cite{ZTH} showed the existence of stationary distribution of CIR process with Markov switching under certain conditions. Besides, Bao and Yuan \cite{BY} generalized the study of long term return of CIR-type models of Deelstra and Delbaen \cite{DD}, Zhao \cite{Zh} to the situation with Markov switching.

In this paper, we develop the study of CIR process with regime-switching in the aspect of analysis of tail property of its stationary distribution after providing sharp conditions to justify its existence and uniqueness. The existence of regime-switching can cause an essential difference in the tail behavior of the stationary distributions of interest rates models. The stationary distribution could be heavy-tailed for CIR process with regime-switching, however, the stationary distribution of CIR process without regime-switching should be always light-tailed. We refer to the monograph of Foss et al. \cite{Foss} on the study of heavy-tailed distributions in probability theory. It is useful to mention that heavy-tailed property is closely related to the long-tailed property of a distribution. According to the works \cite{DY} and \cite{Bar}, the Ornstein-Uhlenbeck process with regime-switching also presents the phenomenon that its stationary distribution could be light-tailed or heavy-tailed. \cite{DY} used the method based on the renewal theory and explicit expression of Ornstein-Uhlenbeck process. \cite{Bar} used the stochastic analysis method. We adopt the idea of \cite{Bar} to study the long time behavior of CIR process with regime-switching in this work. To study its recurrent property, we apply the criteria established by us in \cite{Sh15a} and \cite{SX}. We refer to \cite{Sh15a, SX, YZ} and references therein on the recent development in the study of general regime-switching diffusion processes.

The CIR process with regime-switching investigated in current work is determined by the following stochastic differential equation (SDE):
 \begin{equation}\label{1.1}
 \d r_t=a_{\La_t}(b_{\La_t}- r_t)\d t+2\sigma_{\La_t}\sqrt{r_t}\d B_t,\ \  r_0=x>0,
 \end{equation} where $a_i,\,b_i,\,\sigma_i\in \R$, $\sigma_i\neq 0$, and $(\La_t)$ is a continuous time Markov chain on the state space $\S=\{1,2,\ldots,N\}$ with $2\leq N\leq \infty$ which is independent of Brownian motion $(B_t)$. The transition rate matrix  of $(\La_t)$ is denoted by $Q=(q_{ij})$, which is assumed to be irreducible and conservative. To provide a precise expression of our results, we present the following theorem, which is a partial collection of our Theorems \ref{t2.1} and \ref{t3.2} below.
 \begin{thm}\label{t1.1}
 Assume $a_ib_i\geq 2\sigma_i^2$ for every $i\in \S$, $N<\infty$. Let $(\mu_i)$ be the invariant probability measure of $(\La_t)$. Assume $\sum_{i\in\S}\mu_i a_i>0$. Then $(r_t,\La_t)$ is positive recurrent. Denote by $\pi$ its stationary distribution.
 Let $\kappa$ be defined by \eqref{kappa}, and assume $\kappa>1$.
 \begin{itemize}
   \item[(i)] If $a_{\min}:=\min_{i\in \S} a_i>0$, then there exists some $\delta>0$ such that
   \[\int_{\R_+\times\S} \e^{\delta y}\d \pi<\infty.\]
   \item[(ii)] If $a_{\min}<0$, then
   $\int_{\R_+\times \S}y^p\d \pi<\infty$ if and only if $p\in (0,\kappa)$.
 \end{itemize}
 \end{thm}
 The assertion (i) and (ii) of previous theorem tells us that the fact $a_{\min}>0$ or $a_{\min}<0$ determines the stationary distribution $\pi$  of $(r_t,\La_t)$ being light-tailed or heavy-tailed under some conditions.   Recall the fact on the CIR model without switching:
 \[\d r_t=a(b-r_t)\d t+2\sigma\sqrt{r_t}\d B_t.\]
 When the coefficient $a>0$, this system owns a light-tailed stationary distribution; when the coefficient $a<0$, the system has no stationary distribution. Therefore, the existence of switching can derive the equilibrium between the existence of stationary distribution and tail behavior of stationary distribution. Moreover, for the CIR model \eqref{1.1}, the heavy-tail behaviour of the stationary distribution must imply the existence of regime-switching. This provides a theoretical evidence of existence of regime-switching for interest rates model besides the empirical evidence given in \cite{ABa,ABb, BD}.

The paper is organized as follows. In Section 2, the criteria on justifying the transience and recurrence of CIR process with regime-switching is presented. They are given separately according to the number of states in $\S$ being finite or infinite, and the transition rate matrix being state-independent or state-dependent. The tail behaviour of the stationary distribution of $(r_t,\La_t)$ is investigated in Section 3.

\section{Recurrent property of CIR process with regime-switching}

In this work, we are interested in the situation that the process $(r_t)$ stays strictly positive. Via the connection with the Bessel process, it is known that if $a_i b_i\geq 2\sigma_i^2$, the CIR process in this fixed environment $i\in \S$ is always strictly positive (cf. \cite[Chapter 6.3]{Yor}). We assume that
\begin{itemize}
  \item[(H1)] $a_ib_i\geq 2\sigma_i^2,\qquad \forall\,i\in\S.$
\end{itemize}
Under the condition (H1), it is easy to see that the process $(r_t)$ defined by \eqref{1.1} always stays strictly positive. We refer the readers to \cite{Yor} for more details on the Bessel process and CIR process without switching.

Let $R_t=\sqrt{r_t}$, then under the condition (H1), $R_t$ satisfies the following SDE:
\begin{equation}\label{2.1}
\d R_t=\frac{1}{2R_t} \big(a_{\La_t}b_{\La_t}-\sigma_{\La_t}^2-a_{\La_t} R_t^2\big)\d t+\sigma_{\La_t}\d B_t.
\end{equation}
The recurrent property of $(R_t,\La_t)$ is clearly equivalent to that of $(r_t,\La_t)$, but the diffusion coefficient is not degenerated. In the following we shall use the Lyapunov condition on recurrence established in \cite{Sh15a} for regime-switching diffusion processes to study the recurrent property of $(R_t,\La_t)$. We shall consider first the case $N<\infty$ then the case $N=\infty$.

Note that $(R_t,\La_t)$ is a Markov process, but $(R_t)$ itself is not a Markov process. The infinitesimal generator of $(R_t,\La_t)$ is given by:
\begin{equation}\label{2.2}
\begin{split}
\mathscr A f(x,i)&=L^{(i)} f(\cdot,i)(x)+Qf(x,\cdot)(i)\\
&=\frac{1}2\sigma_i^2\frac{\d^2}{\d x^2} f(x,i)+
\frac{1}{2x} \big(a_ib_i-\sigma_i^2-a_ix^2\big)\frac{\d}{\d x} f(x,i)+ \sum_{j\in\S} q_{ij}f(x,j)
\end{split}
\end{equation} for smooth function $f$ on $\R_+\times \S$.

\begin{thm}\label{t2.1}
Assume (H1) holds and $N<\infty$. Let $(\mu_i)$ be the invariant probability measure of $(\La_t)$. Then
\begin{itemize}
  \item[(i)] if $\sum_{i\in\S}\mu_i a_i>0$, $(R_t,\La_t)$ (hence $(r_t,\La_t)$) is positive recurrent;
  \item[(ii)] if $\sum_{i\in \S}\mu_i a_i<0$, $(R_t,\La_t)$ (hence $(r_t,\La_t)$) is transient.
  \end{itemize}
\end{thm}

\begin{proof}
  Take $h(x)=x^p$ for $p\neq 1$. Then for any $\veps>0$ there exists $M>0$ such that for any $x\geq M$,
  \begin{align*}
    L^{(i)} h(x)&=\frac p2 x^p\big[\big((p-1)\sigma_i^4+a_ib_i-\sigma_i^2\big)/x^2-a_i\big]\\
                &\leq \frac{p}2\big(-a_i+\mathrm{sgn}(p)\veps\big)h(x).
  \end{align*}

  If $\sum_{i\in\S}\mu_i a_i>0$, we take $p>0,\,p\neq 1$ in the definition of $h(x)$. Then there exists $\veps>0$ so that $\sum_{i\in\S}\mu_i(-a_i+\veps)<0$, and hence $\sum_{i\in\S}\frac{p}{2}\mu_i(-a_i+\veps)<0$.
  According to \cite[Theorem 3.1]{Sh15a}, since $h(x)\ra \infty$ as $x\ra \infty$, $(R_t,\La_t)$ is positive recurrent. This implies that $(r_t, \La_t)$ is also positive recurrent in this case.

  If $\sum_{i\in \S}\mu_i a_i<0$, we take $p<0$ in the definition of $h(x)$. Then there exists $\veps>0$ such that $\sum_{i\in\S} \mu_i (-a_i-\veps)>0$. By
  \cite[Theorem 3.1]{Sh15a}, $(R_t,\La_t)$ and hence $(r_t,\La_t)$ is transient due to the fact $h(x)\ra 0$ as $x\ra \infty$.
\end{proof}

\begin{rem}
  Applying previous theorem to the case $N=1$, i.e. a CIR process with no switching, we can get that when $a_i>0$, the corresponding process $(r_t)$ is positive recurrent; when $a_i<0$,  $(r_t)$ is transient. Invoking our discussion in next section on the tail behavior of CIR process with switching, the divergence of the process $(r_t)$ at some environment $i\in\S$ causes the heavy tail property of its stationary distribution. Also the divergence must be restricted by the condition that $\sum_{i\in\S}\mu_i a_i>0$. If not, the process $(r_t)$ will not admit stationary distribution.
\end{rem}

  Next, we proceed to providing two extension of the CIR process \eqref{1.1} with regime-switching. The model $(r_t,\La_t)$ defined by \eqref{1.1} have a meaningful extension that the change of environment $(\La_t)$ can be impacted by the value of $(r_t)$. Namely, the switching rate of the process $(\La_t)$ could be impacted by the process $(r_t)$. Precisely,
  \begin{equation}\label{s-dep}
  \p(\La_{t+\Delta}=j|\La_t=i, r_t=x)=\begin{cases}
     q_{ij}(x)\Delta+o(\Delta),\ \ & i\neq j,\\
     1+q_{ii}(x)\Delta+o(\Delta),\ \ &i=j,
  \end{cases}
  \end{equation}
  provided $\Delta>0$ small enough. Assume that $x\mapsto q_{ij}(x)$ is Lipschitz continuous and $\sup_{x\in \R_+}\sum_{j\neq i} q_{ij}(x)<\infty$ for each $i\in \S$. Then, there is a strong solution  $(r_t,\La_t)$ satisfying \eqref{1.1} and \eqref{s-dep} (cf. \cite{Sh15c}). $(r_t,\La_t)$ is called a state-dependent regime-switching process, whose recurrent property has been studied in \cite{Sh15a} as well. Via the non-singular M-matrix theory, a criterion was established to study the recurrence of CIR process with state-dependent regime-switching.
  
  Before presenting the result on state-dependent regime-switching CIR process, let us recall some useful notation.  
  Let $B$ be a matrix or vector. By $B\geq 0$ we mean that all elements of $B$ are non-negative. By $B\gg 0$ we mean that all elements of $B$ are positive. 
  \begin{defn}[M-Matrix]
    A square matrix $A=(a_{ij})$ is called an M-matrix if $A$ can be expressed in the form $A=s I-B$ with some $B\geq 0$ and $s\geq \mathrm{Ria}(B)$, where $I$ is the identity matrix and $\mathrm{Ria}(B)$ the spectral radius of $B$. When $s>\mathrm{Ria}(B)$, $A$ is called a non-singular M-matrix. 
  \end{defn}
   There are 50 equivalent conditions on the non-singular M-matrix given in the book \cite{BP}. We collect some easily verified conditions below.
   \begin{prop}[\cite{BP}] The following statements are equivalent.
   \begin{enumerate}
     \item $A$ is a non-singular $n\times n$ M-matrix.
     \item All of the principal minors of $A$ are positive, that is 
     \[\begin{vmatrix}
       a_{11}&\ldots&a_{1k}\\ \vdots& &\vdots\\ a_{k1}&\ldots&a_{kk}
     \end{vmatrix} >0 \ \ \text{for every $k=1,2,\ldots,n$}.\]
     \item Every real eigenvalue of $A$ is positive.
   \end{enumerate}
   \end{prop}
   
   We construct an auxiliary Markov chain $(\tilde \La_t)$ on $\S$ with a conservative $Q$-matrix defined by:
   \begin{equation}\label{matrix}
   \tilde q_{ik}=\begin{cases}
     \sup_{x\in\R_+} q_{ik}(x)\ \ &\text{if}\ k<i,\\
     \inf_{x\in\R_+} q_{ik}(x)\ \ &\text{if}\ k>i,
   \end{cases} \quad \text{and}\ \tilde q_{ii}=-\sum_{k\neq i} \tilde q_{ik}.
   \end{equation}
   This auxiliary Markov chain $(\tilde \La_t)$ helps us to control the switching process $(\La_t)$, which is associated with the matrix $H$ defined below. One can reorder the set $\S$ and use above definition to obtain alternative auxiliary Markov chain.
   
   \begin{thm}\label{t-s-dep}
   Let $(r_t,\La_t)$ be defined by \eqref{1.1} and \eqref{s-dep} with $N<\infty$. Assume (H1) holds. Set $\mathrm{diag}(a_1,\ldots,a_N)$ the diagonal matrix with diagonal generated by the vector $(a_1,\ldots,a_N)$. Define the $N\times N$ matrix
   \[H=\begin{pmatrix}
     1& 1&1&\ldots&1\\
     0&1&1&\ldots&1\\
     \vdots&\vdots&\vdots&\ldots&\vdots\\
     0&0&0&\ldots&1
   \end{pmatrix}.\]
   Then, $(R_t,\La_t)$ (and hence $(r_t,\La_t)$) is positive recurrent if there exists some $p>0$ such that the matrix $-\big(\tilde Q-(p/2)\mathrm{diag}(a_1,\ldots,a_N)\big)H$ is a non-singular M-matrix; is transient if there exists some $p<0$ such that the matrix $-\big(\tilde Q-(p/2)\mathrm{diag}(a_1,\ldots,a_N)\big)H$ is a non-singular M-matrix.
   \end{thm}
   
   \begin{proof}
     We can follow the procedure of \cite[Theorem 2.5]{Sh15a} to prove this theorem by taking $V(x)=x^p$, $\beta_i=\frac p 2(-a_i+\mathrm{sgn}(p)\veps)$ and applying the arbitrariness of $\veps$.
   \end{proof}

The recurrent property of Markov chain in finite state space and infinite state space has essential difference. Under the condition that the Makov chain is irreducible, a Markov chain in a finite state space is always recurrent, but the one in the infinite state space may be recurrent or not. The criterion provided in \cite[Theorem 3.1]{Sh15a} depends on the Perron-Frobenius theory of finite order matrix, which has no simple extension to deal with infinite order matrix. To deal with the regime-switching processes with switching in an infinite state space, J. Shao has raised two methods in \cite{Sh15a} and \cite{Sh15b}: finite dimensional partition method and principle eigenvalue method, which have been extended to deal with the stability problem in \cite{SX}. The finite partition method is based on the theory of non-negative M-matrix, and see
\cite[Theorem 2.7]{Sh15a} for more details.

Now we proceed to studying the recurrent property of $(R_t,\La_t)$ when $(\La_t)$ is a Markov chain in an infinite state space using the method of principal eigenvalue method. Suppose that $(\La_t)$ is reversible with invariant probability measure $(\mu_i)$.
Set $L^2(\mu)=\{f; \sum_{i=1}^\infty \mu_if_i^2<\infty\}$, and denote by $\|\cdot\|$ and $\la \cdot,\cdot\raa$ its associated norm and inner product.
Set
\begin{equation}\label{2.3}
D(f)=\frac 12\sum_{i,j=1}^\infty \mu_i q_{ij}(f_j-f_i)^2+\sum_{i=1}^\infty \mu_i a_i f_i^2,\quad f\in L^2(\mu),
\end{equation}
and
\begin{equation}\label{2.4}
\tilde D(f)=\frac 12\sum_{i,j=1}^\infty \mu_iq_{ij}(f_j-f_i)^2+\frac 12\sum_{i=1}^\infty \mu_i a_i f_i^2, \quad f\in L^2(\mu).
\end{equation}
Note that $D(f)$ and $\tilde D(f)$ are not necessary Dirichlet forms since $a_i$ may be negative for some $i\in\S$. The principle eigenvalues $\lambda_0$ and $\tilde \lambda_0$ corresponding to $D(f)$ and $\tilde D(f)$ respectively are defined by
\begin{equation}\label{2.5}
\lambda_0=\inf\{D(f);\,\|f\|=1\},\quad \tilde \lambda_0=\inf\{\tilde D(f);\,\|f\|=1\}.
\end{equation}

\begin{thm}\label{t2.2}
Assume (H1) holds and $N=\infty$. Suppose $(\La_t)$ is reversible and positive recurrent with invariant probability measure $(\mu_i)$.
\begin{itemize}
  \item[(i)] Suppose there exists a bounded function $(g_i)_{i\in\S}$ such that $D(g)=\lambda_0\|g\|^2$, $\liminf_{i\ra \infty} g_i\neq  0$. If $\lambda_0>0$, then $(R_t,\La_t)$ is recurrent.
  \item[(ii)] Assume that $\tilde \lambda_0>0$ and $\tilde\lambda_0$ is attainable, i.e. there exists $g\in L^2(\mu)$, $g\not\equiv 0$ so that $\tilde D(g)=\tilde \lambda_0 \|g\|^2$, then $(R_t,\La_t)$ is transient.
\end{itemize}
\end{thm}

\begin{proof}
  (i)\ The idea of this argument comes from \cite[Theorem 4.2]{SX}. By variational method (cf. \cite[Theorem 3.2]{SX}), it holds that $Qg(i)=-\lambda_0 g_i$, and $g_i>0$ for each $i\in \S$. Let
  $V(x,i)=g_i x^2$, then there exists $M_0>0$ such that $\dis \frac{\sigma_i^4+a_ib_i-\sigma_i^2}{x^2}<\frac{\lambda_0}{2}$ for any $x>M_0$. Moreover,
  \begin{equation}\label{2.6}
  \begin{split}
    \mathscr A V(x,i)&\leq \Big( Qg(i) -a_ig_i+\frac{\sigma_i^4+a_ib_i-\sigma_i^2}{x^2} g_i\Big) x^2\\
    &\leq (-\lambda_0+\frac{\lambda_0}{2}) g_i x^2=-\frac{\lambda_0}{2} V(x,i)<0.
  \end{split}
  \end{equation}
  To study the recurrence of $(R_t,\La_t)$, we only need to consider the situation $R_0=x_0>M_0$.
  Set
  \[\tau=\inf\{t>0; (R_t,\La_t)\in \{x:\,x\leq K\}\times \{1,\ldots,m_0\}\},\]
  where constant $K$ satisfies $x_0>K>M_0$ and $m_0\in \N$ satisfies $\La_0=\ell>m_0$. Put
  \[\tau_K=\inf\{ t>0; R_t\geq K\}.\]
  Applying It\^o's formula to $(R_t,\La_t)$, by \eqref{2.6}, we get
  \begin{align*}
    \E V(R_{t\wedge \tau\wedge \tau_K},\La_{t\wedge \tau\wedge \tau_K})= V(x_0,\ell)+\E\int_0^{t\wedge \tau\wedge \tau_K}\mathscr A V(R_s,\La_s)\d s\leq V(x_0,\ell),
  \end{align*}
  which yields
  \[\p(\tau>\tau_K)\leq \frac{V(x_0,\ell)}{K^2\inf_{i\in\S} g_i}.\]
  Letting $K\ra \infty$, it follows immediately from $\lim_{K\ra\infty}\tau_K= \infty$ a.e. that
  $\p(\tau=\infty)=0$. Consequently, $\p(\tau<\infty)=1$ and $(R_t,\La_t)$ is recurrent.

  The proof of (ii) is similar to that of \cite[Theorem 4.2]{SX}, and hence is omitted.
\end{proof}

As an application of Theorem \ref{t2.1}, we go to present a result on the long term returns of interest rates, which develops the corresponding results of Deelstra and Delbaen \cite{DD} on an extended CIR model without regime-switching. Following \cite{DD}, the work \cite{ZTH} studied extended CIR model with regime-switching under a strong condition. Restricted to classical CIR model \eqref{1.1}, their condition means that $a_i>0$ for all $i\in\S$, i.e. the corresponding CIR process in each fixed environment $i$ is recurrent. Especially, under the conditions of \cite{ZTH}, there is no heavy-tailed phenomenon appeared for the stationary distribution. Our method and corresponding results can be easily used to study extended CIR process considered in \cite{ZTH}. Under the existence of stationary distribution given in Theorem \ref{t2.1}, according to the strong ergodicity theorem (see, for instance, \cite[Theorem 4.4]{YZ}), the following assertion holds.
\begin{cor}\label{t2.3}
Assume that $(r_t,\La_t)$ is positive recurrent. Then for any measurable function $f$ on $\R_+\times \S$ such that
\[\sum_{i\in\S} \int_{\R_+} |f(x,i)|\pi(\d x,i)<\infty,\] it holds
\begin{equation}\label{2.7}
\lim_{t\ra\infty}\frac 1t\int_0^t f(r_s,\La_s)\d s=\sum_{i\in \S} \int_{\R_+} f(x,i)\pi(\d x,i),
\end{equation} where $\pi$ denotes the stationary distribution of $(r_t,\La_t)$.
\end{cor}

It is well known that CIR process is closely related to the Bessel process
(cf. \cite[Chapter 6]{Yor}). In the following, we are also interested to establish a connection between CIR process with regime-switching and squared Bessel process with regime-switching.
\begin{prop}\label{t2.4}
Let $(r_t,\La_t)$ be the solution of \eqref{1.1}, then
\begin{equation}\label{2.8}
(r_t)_{t\geq 0}=\Big(\frac{1}{\ell(t)}\rho\Big(\int_0^t\ell(s)\d s\Big)\Big)_{t\geq 0},
\end{equation}
in the sense that both sides admit the same distribution, where $\ell(t)=\exp\big(\int_0^t a_{\La_u}\d u\big)$, and $(\rho_t)$ is a squared Bessel process with regime-switching satisfying the SDE:
\begin{equation}\label{2.9}
\d \rho_t=a_{\La(A_t)} b_{\La(A_t)}\d t+\sigma_{\La(A_t)} \sqrt{\rho_t}\d W_t,
\end{equation} where
$A_t=\inf\{u:\int_0^u\ell(s)\d s=t\}$.
\end{prop}

\begin{proof}
  We follow the idea of \cite[Theorem 6.3.5.1]{Yor} to prove this proposition.
  Set $C(t)=\int_0^t \ell(s)\d s$, $t>0$, then $t\mapsto C(t)$ is a strictly increasing process and so its inverse $A(t)=\inf\{u: C(u)=t\}$ is well defined.

  Let $Z_t=r_t\ell(t)$, then It\^o's formula yields that
  \begin{equation}\label{2.10}
    \d Z_t= a_{\La_t}b_{\La_t}\ell(t)\d t+ \sigma_{\La_t}\sqrt{\ell(t)}\sqrt{Z_t}\d B_t.
  \end{equation}
  Set $0=\tau_0<\tau_1<\ldots<\tau_k<\ldots$ be the sequence of jumping times of the Markov chain $(\La_t)$. Put $\zeta_k=C(\tau_k)$, $k\geq 0$, then $A_{\zeta_k}=\tau_k$ and there is no jumps for $(\La_t)$ in the interval $[A_{\zeta_k},A_{\zeta_{k+1}})$.
  Set $n(t)=\inf\{k;\,\zeta_k\leq t<\zeta_{k+1}\}$ for $t\geq 0$, and for the convenience of notation, denote by $\zeta_{n(t)+1}=t$.

  By \eqref{2.10},
  \begin{align*}
    Z_{A(\zeta_{k+1})}&=Z_{A(\zeta_k)}+\int_{A(\zeta_{k})}^{A(\zeta_{k+1})} a_{\La_s}b_{\La_s}\ell(s)\d s +\int_{A(\zeta_{k})}^{A(\zeta_{k+1})} \sigma_{\La_s}\sqrt{\ell(s)Z_s}\d B_s\\
    &=Z_{A(\zeta_k)}+a_{\La_{A(\zeta_k)}}b_{\La_{A(\zeta_k)}}
    (\zeta_{k+1}-\zeta_k) +\sigma_{\La_{A(\zeta_k)}}\int_{ A(\zeta_{k})}^{A(\zeta_{k+1})}
    \sqrt{\ell(s)Z_s}\d B_s.
  \end{align*}
   Note that $\int_0^t\sqrt{\ell(s)Z_s}\d B_s$ is a local martingale with quadratic variation $\int_0^{A(t)}\ell(s)Z_s\d s=\int_0^tZ_{A(u)}\d u$. Thus, there is a Brownian motion $(W_t)$ such that
   \[\int_0^{A(t)}\sqrt{\ell(s)Z_s}\d B_s=\int_0^t \sqrt{Z_{A(s)}}\d W_s.\]
   Consequently,
   \begin{align*}
     Z_{A(t)}&=Z_{A(0)}+\sum_{k=0}^{n(t)}\big(Z_{A(\zeta_{k+1})}-Z_{A(\zeta_k)}\big)\\
     &=Z_{A(0)}+\sum_{k=0}^{n(t)}\int_{\zeta_k}^{\zeta_{k+1}} a_{\La_{A(s)}}b_{\La_{A(s)}}\d s+\sum_{k=0}^{n(t)}\int_{\zeta_k}^{\zeta_{k+1}}\sigma_{\La_{A(s)}}\sqrt{Z_{A(s)}}\d W_s\\
     &=Z_{A(0)}+\int_0^ta_{\La_{A(s)}}b_{\La_{A(s)}}\d s+\int_0^t\sigma_{\La_{A(s)}}\sqrt{Z_{A(s)}}\d W_s.
   \end{align*}
   Let $\rho_t=Z_{A(t)}$, then $(\rho_t)_{t\geq 0}$ satisfies the SDE \eqref{2.9} and $r_t=\frac{1}{\ell(t)} Z_t=\frac{1}{\ell(t)} \rho_{C(t)}$, which is the desired result.
\end{proof}

\section{Long time behavior of CIR process with switching}

In this section, we only consider the CIR process with switching in a finite state space, i.e. $N<\infty$. We focus on difference of the tail behavior of the stationary distribution if it exists caused by the existence of regime-switching. Our result reveals from the theoretical point of view the existence of regime-switching in interest rates as shown in \cite{ABa,ABb,BD} by empirical evidence.

Denote $Q_p$ the matrix $Q-p\mathrm{diag}(a_1,\ldots,a_N)$, where $\mathrm{diag}(a_1,\ldots,a_N)$ stands for the diagonal matrix with diagonal $(a_1,\ldots,a_N)$. Set
\[\eta_p=-\max_{\gamma\in \mathrm{spec}(Q_p)} \mathrm{Re}\, \gamma,\]
where $\mathrm{spec}(Q_p)$ denotes the spectrum of operator $Q_p$.
According to Propositions 4.1 and 4.2 in \cite{Bar}, it holds:
\begin{lem}\label{t3.1}
(i)\ For any $p>0$, there exist $0<C_1(p)<C_2(p)<\infty$ such that, for any initial distribution $\nu_0$ on $\S$, any $t>0$,
\[C_1(p)\e^{-\eta_p t}\leq \E_{\nu_0} \Big[\e^{-\int_0^t p a_{\La_s} \d s}\Big]\leq C_2(p)\e^{-\eta_pt}.\]\\
(ii)\ Set $a_{\min}=\min\{a_i; i\in \S\}$. If $a_{\min}\geq 0$, then $\eta_p>0$ for all $p>0$; if $a_{\min}<0$, there exists $\kappa\in (0, \min\{ -q_i/a_i; a_i<0\})$ such that $\eta_p>0$ for $p<\kappa$ and $\eta_p<0$ for $p>\kappa$.
\end{lem}

\begin{thm}\label{t3.2}
Assume that (H1) holds, $\sum_{i\in \S} \mu_i a_i>0$, and $N<\infty$. Let $\pi$ be the stationary distribution of $(r_t,\La_t)$ on $\R_+\times \S$. Set
\begin{equation}\label{kappa}\kappa=\sup\{p>0; \eta_p>0\}\in(0,\infty].
\end{equation}
Assume $\kappa>1$. Then
\begin{itemize}
  \item[(i)] If $a_{\min}>0$, then for some $\delta>0$,
  $\dis \int_{\R_+} \e^{\delta y}\pi(\d y, \S)<\infty$.
  \item[(ii)] If $a_{\min}<0$, then the $p^{th}$ moment of $\pi$ is finite if and only if $0<p<\kappa$.
\end{itemize}
\end{thm}

\begin{proof}
(i) Set $\alpha=\max_{i\in \S} \sigma_i^2/a_i$.
  For $0<\delta<1/(2\alpha)$,
  \begin{equation}\label{3.6}
  \d \e^{\delta r_t}= \delta \e^{\delta r_t}\Big[a_{\La_t}b_{\La_t}-a_{\La_t}\big(1-2\delta\frac{\sigma_{\La_t}^2}{a_{\La_t}}\big)r_t\Big]\d t+2\delta \e^{\delta r_t}\sigma_{\La_t}\sqrt{r_t}\d B_t.
  \end{equation}
  For any $c>0$, there exists a constant $M>0$ such that, for any $x\in \R_+$, any $i\in \S$,
  \[a_ib_i-a_{\min}(1-2\delta \alpha)x\leq -c +M\e^{-\delta x}.\]
  Therefore, taking expectation in both sides of \eqref{3.6} deduces that
  \[\frac{\d \E[\e^{\delta r_t}]}{\d t}\leq M\delta -c\delta \E[\e^{\delta r_t}],\]
  which implies that
  \[\E[\e^{\delta r_t}]\leq \E[\e^{\delta r_0}]\e^{-c\delta t}+M\delta \int_0^t\e^{-c\delta(t-u)}\d u.\]
  Hence,
  \[\sup_{t>0}\E[\e^{\delta r_t}] \ \text{is finite when $\dis \E[\e^{\delta r_0}]$ is finite},\]
  which implies that
  \[\int_{\R_+} \e^{\delta y}\pi(\d y, \S)<\infty \ \text{for $\delta<1/(2\alpha)$.}\]
  This means that when $a_{\min}>0$, the stationary distribution $\pi$ of $(r_t,\La_t)$ is light-tailed.

  (ii)\  For the case  $p\in (1,\kappa)$,
  \begin{equation}\label{3.1}
    \d r_t^p=\big[ -pa_{\La_t} r_t^p+p(a_{\La_t}b_{\La_t}+2(p-1)\sigma_{\La_t}^2)r_t^{p-1}\big]\d t+2p\sigma_{\La_t} r_t^{p-1}\sqrt{r_t} \d B_t.
  \end{equation}
  Denote by $\mathscr F_T^{\La}=\sigma\big\{\La_s; 0\leq s\leq T\big\}$ for $T>0$. By the independence of $(\La_t)$ and $(B_t)$, taking expectation in \eqref{3.1} conditioning on $\mathscr F_T^{\La}$ leads to
  \begin{equation*}
    \alpha_p'(t)=-pa_{\La_t}\alpha_p(t)+p(a_{\La_t}b_{\La_t}+2(p-1)\sigma_{\La_t}^2)\alpha_{p-1}(t), \quad t\in(0,T].
  \end{equation*}
  For any $\veps>0$ there exists $c>0$ such that
  \[\alpha_p'(t)\leq (-pa_{\La_t}+\veps)\alpha_p(t)+c,\]
  which deduces that
  \begin{equation}\label{3.2}
  \alpha_p(t)\leq \alpha_p(0)\e^{\int_0^t (-pa_{\La_s}+\veps)\d s}+c\int_0^t\e^{\int_u^t (-pa_{\La_s}+\veps)\d s}\d u.
  \end{equation}
  Taking expectation in both sides of \eqref{3.2} and using Lemma \ref{t3.1}, we obtain
  \[\E r_t^p\leq \E r_0^p C_2(p)\e^{(-\eta_p+\veps)t}+cC_2(p)\int_0^t \e^{(-\eta_p+\veps)(t-u)}\d u.\]
  When $\eta_p>\veps>0$, this implies
  \begin{equation}\label{3.4}
  \sup_{t>0} \E r_t^p<\infty, \quad \text{if $\E r_0^p<\infty$}.
  \end{equation}
  Put $P_t$ the semigroup associated with the Markov process $(r_t,\La_t)$. By Theorem \ref{t2.1}, the condition $\sum_{i\in\S}\mu_i a_i>0$ yields that for $\delta_{(x,i)}P_t$ converges weakly to $\pi$ as $t\ra\infty$. Namely, the distribution of $(r_t,\La_t)$ with $(r_0,\La_0)=(x,i)$ converges weakly to its stationary distribution $\pi$ as $t\ra \infty$. Thus, \eqref{3.4} implies that
  \begin{equation}\label{3.5}
  \int_{\R_+}y^p\pi(\d y,\S)\leq \liminf_{t\ra \infty}\int_{\R_+\times \S} y^p \d\big(\delta_{(x,i)}P_t\big)=\liminf_{t\ra \infty} \E r_t^p<\infty,
  \end{equation}
  if $\eta_p>0$ for $p>0$. Combining with Lemma \ref{t3.1}, if $a_{\min}\geq 0$, then $\eta_p>0$ for all $p>0$, and further $\pi$ owns finite moment of all orders; if $a_{\min}<0$, then for $1\leq p<\kappa$, $\eta_p>0$ and further $\pi$ owns finite $p^{th}$ moment.

  To show the moment of order $\kappa$ of $\pi$ is infinite, we use the proof by contradiction. Assume $\int_{\R_+} y^\kappa\pi(\d y,\S)<\infty$, then we take the initial distribution of $(r_0,\La_0)$ to be the stationary distribution $\pi$, which implies that the distribution of $(r_t,\La_t)$ is also $\pi$ for $t>0$. However, It\^o's formula yields that
\begin{align*}
  \alpha_\kappa(t)&=\int_0^t\e^{-\kappa\int_u^t a_{\La_s}\d s}\kappa\big(a_{\La_u}b_{\La_u}+2(\kappa-1)\sigma_{\La_u}^2\big)\alpha_{\kappa-1}(u)\d u+\alpha_\kappa(0)\e^{-\kappa\int_0^t a_{\La_s}\d s}\\
  &\geq \int_0^t\e^{-\kappa\int_u^t a_{\La_s}\d s} \kappa\min_{i\in\S}\{a_ib_i+2(\kappa-1)\sigma_i^2\} \alpha_{\kappa-1}(u)\d u.
\end{align*}
Since $\eta_\kappa=0$, Lemma \ref{t3.1} shows that
\[\e^{-p\int_0^ta_{\La_s}}\d s\geq C_1(p).\]
Hence,
\[\int_{\R_+}y^\kappa \pi(\d y, \S)=\E_{\pi} r_t^p=\E\alpha_p(t)\geq \int_0^t C_1(p)p\min_{i\in\S}\{a_ib_i+2(p-1)\sigma_i^2\}\E_{\pi} r_s^{p-1}\d s,\quad t>0,\]
where we use $\E_{\pi}$ to emphasize the initial distribution of $(r_t,\La_t)$ is $\pi$.
Note that $\E_{\pi} r_t^{\kappa-1}=\int_{R_+} y^{\kappa-1}\pi(\d y, \S)\in (0,\infty)$, then letting $t$ go to $+\infty$ in previous inequality leads to
$\int_{\R_+} y^\kappa \pi(\d y,\S)=\infty$, which is contradict to our assumption. Therefore, the $\kappa$-th moment of $\pi$ is infinite.
\end{proof}

\begin{rem}\label{rem-3.1}
In the last step of previous argument, note that merely using the fact $\delta_{(x,i)}P_t$ weakly converges to $\pi$, one can not derive $\int_{\R_+}y^p\pi(\d  y,\S)=\infty$ from that $\lim_{r\ra \infty} \E r_t^p =\infty$. We present a simple example below.
Let \[\mu_n(\d x)=\begin{cases}
   1-\frac{1}{\sqrt{n}}, & x\in \big[0,\frac{\sqrt{n- n^{-1}}}{\sqrt{n}-1}\big],\\
   \frac 1{\sqrt{n}}, &\big[n,n+\sqrt{n}-\sqrt{n-n^{-1}}\big],
\end{cases}\]
and $\mu(\d x)=1 $ if $x\in[0,1]$; $\mu(\d x)=0$, otherwise. It is easy to check that for any bounded continuous function $f$ on $[0,\infty)$,
\[\lim_{n\ra \infty}\int_{\R_+} f(x)\mu_n(\d x)=\int_{\R_+} f(x)\mu(\d x),\]
which means that $\mu_n$ weakly converges to $\mu$ as $n\ra \infty$.
On the contrary,
\begin{align*}
&\lim_{n\ra \infty} \int_{\R_+}x^3\mu_n(\d x)\\
&=\lim_{n\ra \infty}\frac 1{4\sqrt{n}} \big((n\!+\!\sqrt n\!-\!\sqrt{n\!-\!n^{-1}})^4-n^4\big)+\frac 14 \big(1\!-\!\frac{1}{\sqrt{n}}\big)\Big(\frac{\sqrt{n-n^{-1}}}{\sqrt n-1}\Big)^4=\infty.
\end{align*}
But
\[ \int_{\R_+} x^3\mu(\d x)<\infty.\]
\end{rem}

To deal with the case $\kappa\in (0,1]$, the methodology used in \cite{Bar} by replacing the function $y\mapsto |y|^p$ by $f: y\mapsto \frac{|y|^{p+2}}{1+y^2}$ to exploit the case $p\in (0,2)$ therein is actually not applicable. The reason is that the signs of coefficients of $f'(Y_t)$ vary according to the state $i$ in $\S$, and this causes the difficulty to control the quantity $Y_t\, f'(Y_t)$ by $f(Y_t)$. In the following, we exploit the long time behaviour of $(r_t,\La_t)$ when $\kappa\in (0,1)$ under the  boundedness of $\E r_t^{-1}$ for $t>0$.

\begin{thm}\label{t3.3}Suppose $N<\infty$, $\sum_{i\in\S} \mu_i a_i>0$, $a_{\min}<0$ and $\kappa\in (0,1]$.
Set $\tilde Q_1=Q-\mathrm{diag}(\tilde a_1,\ldots,\tilde a_N)$ where $\tilde a_i=-a_i$, $i\in \S$.
Put
\[\tilde \eta_1=-\max_{\gamma\in \mathrm{spec}(\tilde Q_1)} \mathrm{Re}\, \gamma.\]
Assume $a_ib_i\geq 4\sigma_i^2$ for all $i\in \S$ and $\tilde \eta_1>0$. Then the following assertions hold.
\begin{itemize}
\item[(i)] For any $p\in (0,\kappa)$, $\sup_{t>0} \E r_t^p<\infty$ if $\E r_0^p<\infty$, and further $\int_{\R_+} y^p\pi(\d y,\S)<\infty$.
\item[(ii)] For any $p\geq \kappa$, $\sup_{t>0}\E r_t^p=\infty$ and $\int_{\R_+} y^\kappa\pi(\d y,\S)=\infty$.
\end{itemize}
\end{thm}

\begin{proof}
  (i)\ When $\kappa\in (0,1]$, instead of the condition (H1), we need to assume a stronger condition that $a_ib_i\geq 4\sigma_i^2$ for all $i\in \S$.
  Then, for $p\in (0,\kappa)$, we obtain
  \begin{equation}\label{3.7}
  \d r_t^p=\big[-pa_{\La_t}r_t^p+p(a_{\La_t}b_{\La_t}+2(p-1)\sigma_{\La_t}^2)r_t^{p-1}\big]\d t+2p\sigma_{\La_t}r_t^{p-1}\sqrt{r_t}\d B_t.
  \end{equation}
  We need to justify the finiteness of $\E r_t^{p-1}$ for $p\in (0,\kappa)$. For this purpose, we only need to show the boundedness of $\E r_t^{-1}$.
  It holds
  \begin{align*}
  \d r_t^{-1}&=[a_{\La_t} r_t^{-1} -(a_{\La_t}b_{\La_t}-4\sigma_{\La_t}^2)r_t^{-2}]\d t-2\sigma_{\La_t} r_t^{-2}\sqrt{r_t}\d B_t\\
  &\leq -\tilde a_{\La_t} r_t^{-1}\d t-2\sigma_{\La_t} r_t^{-2}\sqrt{r_t}\d B_t.
  \end{align*}
  By Lemma \ref{t3.1} and $\tilde \eta_1>0$,
  \[\E r_t^{-1}\leq \E r_0^{-1}\e^{-\int_0^t \tilde a_{\La_s}\d s} \leq \E r_0^{-1}C_2(1)\e^{-\tilde \eta_1 t},
  \] and hence $\sup_{t>0} \E r_t^{-1} <\infty$ if $\E r_0^{-1}<\infty$. Furthermore, setting $(r_0,\La_0)=(x,i)\in (0,\infty)\times\S$, it follows immediately that
   \begin{equation}\label{3.8}
   \int_{\R_+} y^{-1}\pi(\d y, \S)\leq \liminf_{t\ra \infty} \E r_t^{-1}<\infty.
   \end{equation}
  Invoking \eqref{3.7}, there exists a constant $c>0$ such that
  \begin{align*}
    \alpha_p'(t) \leq -p a_{\La_t} \alpha_p(t)+ c,
  \end{align*}which yields that
  \[\alpha_p(t)\leq \alpha_p(0)\e^{\int_0^t -p a_{\La_s}\d s} +c\int_0^t\e^{\int_u^t-pa_{\La_s}\d s}\d u.\]
  Similar to \eqref{3.2} and \eqref{3.4}, the fact $\eta_p>0$ for $p\in (0,\kappa)$ implies that
  \[\sup_{t>0} \E r_t^p<\infty,\ \text{if}\ \E r_0^p<\infty.\]
  Hence
  \[\int_{R_+} y^p\pi(\d y, \S)\leq \liminf_{t\ra \infty} \E r_t^p<\infty.\]

  (ii)\ When $p=\kappa$, the equation \eqref{3.7} implies that
  \[\E r_t^\kappa\geq \E\int_0^t\e^{-\kappa\int_u^t a_{\La_s}\d s} \kappa\min_{i\in\S} \{a_i b_i+2(\kappa-1)\sigma_i^2\}\E[r_u^{\kappa-1}\big|\mathscr F_T^{\La}]\d u.\]
  Due to Lemma \ref{t3.1}, \eqref{3.8}, if $\int_{\R_+} y^\kappa \pi(\d y,\S)<\infty$, we take the initial distribution of $(r_t,\La_t)$ to be $\pi$ to yield that
  \[\int_{\R_+}y^\kappa \pi(\d y,\S)=\E_{\pi} r_t^\kappa\geq \kappa C_1(\kappa)\min_{i\in \S} \{a_i b_i+2(\kappa-1)\sigma_i^2\} t \int_{\R_+} y^{\kappa-1} \pi(\d y,\S),\]
  and hence $\int_{\R_+} y^\kappa \pi(\d y,\S)=\infty$ by letting $t\ra \infty$ in previous inequality. This is contradict to our assumption that $\int_{\R_+} y^{\kappa} \pi(\d y,\S)<\infty$, and we get the desired conclusion.
\end{proof}

\section{Conclusion}
We mainly focus on the tail behavior of stationary distributions of the CIR processes with regime-switching.  We provide an explicit connection between the heavy-tailed property of its stationary distribution and  the coefficients of CIR process with regime-switching. According to this result, heavy-tailed property of interest rates implies that the application of CIR model with regime-switching is more suitable in practice than CIR model without regime-switching. Further investigation on the quantitative property, such as first passage probability, of CIR process with regime-switching will be considered to learn more impact on the existence of regime-switching in the interest rates models. Besides, there is difficulty to estimate the development of the functional $\int_0^t f(\La_s)\d s$ for $f:\S\ra \R$ when $(\La_t)$ is a state-dependent switching process. Hence, at present stage, we cannot provide more information on the tail behavior of the stationary distribution for state-dependent regime-switching CIR process. More work is needed in this direction.

\end{document}